\documentclass[letterpaper,10pt]{amsart}
\usepackage{indentfirst} % Indent first paragraph after section header
\usepackage{amssymb}
\usepackage{amsmath}
\usepackage{mathabx} % makes many symbols (as $\leq$) more beautiful
\usepackage{amsthm}
\usepackage{thmtools}
\usepackage{bbm}             %to use $\mathbbm{1}$
\usepackage[usenames,dvipsnames]{xcolor} % With XCOLOR, printed quality is good, (but with COLOR it's bad)
\usepackage{comment}

\newtheorem{theorem}{Theorem}[section]
\newtheorem{definition}[theorem]{Definition}
\newtheorem{proposition}[theorem]{Proposition}
\newtheorem{corollary}[theorem]{Corollary}
\newtheorem{lemma}[theorem]{Lemma}

\newtheorem{remark}[theorem]{Remark}

\newtheorem*{theorem-non}{Theorem}

\declaretheorem[name=Acknowledgements,numbered=no]{ack}

\theoremstyle{definition}

\newcommand{\Dim}{\textmd{dim}}

\newcommand{\R}{\mathbb{R}}

\newcommand{\Z}{\mathbb{Z}}

\newcommand{\vol}{\text{vol}}
\DeclareMathOperator*{\esssup}{ess\,sup}
\DeclareMathOperator*{\essinf}{ess\,inf}

\begin{document}

\title{Ruelle's inequality in negative curvature} %Use the shortened version of the full title

\author[F. Riquelme]{Felipe Riquelme}
\date{\today}
\thanks{F.R. was supported by Programa Postdoctorado FONDECYT Proyecto 3170049}
\address{IMA, Pontificia Universidad Cat\'olica de Valpara\'iso, Blanco Viel 596, Cerro Bar\'on, Valpara\'iso, Chile.}

\email{friquelme.math@gmail.com}

% It is required to enter 2010 MSC.
%\subjclass{Primary: 37A05, 37A35, 37D40; Secondary: 37D10, 37D35}
% Please provide minimum  5 keywords.
%\keywords{Entropy, Lyapunov exponents, Geodesic flow, Gibbs measure, Geometric potential.}

\maketitle

\begin{abstract} In this paper we study different notions of entropy for measure-preserving dynamical systems defined on noncompact spaces. We see that some classical results for compact spaces remain partially valid in this setting. We define a new kind of entropy for dynamical systems defined on noncompact Riemannian manifolds, which satisfies similar properties to the classical ones. As an application, we prove Ruelle's inequality and Pesin's entropy formula for the geodesic flow in manifolds with pinched negative sectional curvature.
\end{abstract}

\section{Introduction}

\subsection{Motivation and statements of main results}

\emph{Ruelle's inequality} \cite{Ruelle} is an important result in ergodic theory for smooth dynamical systems relating two fundamental concepts: \emph{measure-theoretic entropy} and \emph{Lyapunov exponents}. It precisely states that if $f:M\to M$ is a $C^1$-diffeomorphism of a compact Riemannian manifold and $\mu$ is an $f$-invariant probability measure on $M$, then the measure-theoretic entropy $h_\mu(f)$ is bounded from above by the sum of the positive Lyapunov exponents, i.e.
\begin{equation}\label{eq:Ruelle}
h_\mu(f) \leq \int \sum_{\lambda_{j}(x)>0} \lambda_{j}(x)\dim(E_{j}(x)) d\mu(x),
\end{equation}
where $\{\lambda_j(x)\}$ is the set of Lyapunov exponents at $x\in M$ and $\dim(E_{j}(x))$ is the multiplicity of $\lambda_j(x)$.

Once inequality \ref{eq:Ruelle} is established, the question about the equality case, known as \emph{Pesin's entropy formula}, arises naturally. For $C^{1+\alpha}$-diffeomorphisms F. Ledrappier and L.-S. Young showed in \cite{LY} that an $f$-invariant probability measure verifies Pesin's entropy formula if and only if it is absolutely continuous along unstable manifolds (see also \cite{Pesin}, \cite{LS} and  \cite{Ledrappier}).

Surprisingly, Ruelle's inequality can fail to be true on noncompact manifolds. To be more precise, in \cite{Riquelme1} we proved that there exist smooth dynamical systems having (arbitrary) positive measure-theoretic entropy whereas there are no positive Lyapunov exponents. This implies in particular that, even for smooth enough dynamical systems, Ruelle's inequality is not always verified when the manifold is not compact. Therefore, it becomes an important question to investigate whether or not this inequality is guaranteed in noncompact situations.

The aim of this paper is to prove that Ruelle's inequality and Pesin's entropy formula are verified for the geodesic flow on manifolds with negative sectional curvature. To be more precise, we show

\begin{theorem}\label{thm1} Let $X$ be a complete Riemannian manifold with dimension at least 2 and pinched negative sectional curvature. Assume that the derivatives of the sectional curvature are uniformly bounded. Then, for every $(g^t)$-invariant probability measure $\mu$ on $T^{1}X$, we have
$$h_{\mu}(g)\leq\int\sum_{\lambda_{j}(v)>0}\lambda_{j}(v)\dim(E_{j}(v))d\mu(v).$$
\end{theorem}

The proof of Theorem \ref{thm1} is based in three key facts. First, the geometric potential determines the sum of the positive Lyapunov exponents. Second, the Lebesgue measure on $T^1X$ verifies a Gibbs property for the geometric potential (see Proposition \ref{gibbsprop}). And third, the measure-theoretical entropy is bounded from above by the exponential decay of the volume of dynamical balls on compact sets (see Corollary \ref{cor0}). 

As said before, we also treat the equality case on Ruelle's inequality. No additional assumptions to those of Theorem \ref{thm1} are needed.

\begin{theorem}\label{thm2} Let $X$ be a complete Riemannian manifold with dimension at least 2 and pinched negative sectional curvature. Assume that the derivatives of the sectional curvature are uniformly bounded. Let $\mu$ be a $(g^{t})$-invariant probability measure on $T^{1}X$. Then $\mu$ has absolutely continuous conditional measures on unstable manifolds if and only if
$$h_{\mu}(g)=\int \sum_{\lambda_{j}(v)>0} \lambda_{j}(v)\dim(E_{j}(v)) d\mu(v).$$
\end{theorem}

Finally, using tools of thermodynamical formalism (see last section), we can improve Theorem \ref{thm1} and Theorem \ref{thm2} as follows

\begin{corollary}\label{cor:summary_Ruelle} Let $X$ be a complete Riemannian manifold with dimension at least 2 and pinched negative sectional curvature. Assume that the derivatives of the sectional curvature are uniformly bounded. 
\begin{enumerate}
\item[(i)] If $X$ has finite volume and $\nu$ denotes the normalized Lebesgue measure on $T^1X$, then
$$h_{\nu}(g)=\int \sum_{\lambda_{j}(v)>0} \lambda_{j}(v)\dim(E_{j}(v)) d\nu(v).$$
Moreover, any other $(g^{t})$-invariant probability measure $\mu$ on $T^{1}X$ verifies
$$h_{\mu}(g)<\int \sum_{\lambda_{j}(v)>0} \lambda_{j}(v)\dim(E_{j}(v)) d\mu(v).$$
\item[(ii)] If $X$ has infinite volume, then there is no $(g^{t})$-invariant probability measure on $T^1X$ verifying Pesin's entropy formula.
\end{enumerate}
\end{corollary}

\subsection{Structure of the paper} In Section 2 we are devoted to study deeply the concept of entropy for arbitrary topological dynamical systems. We are parti\-cularly interested in the interaction between classical notions of entropy such as measure-theoretic entropy, local entropy and $\delta$-entropies. For dynamical systems defined on Riemannian manifolds we introduce a geometric notion of entropy, called Riemannian local entropy, that roughly speaking measures the exponential decay of the volume of dynamical balls. At the end of the section we prove that the measure-theoretic entropy is less than the Riemannian local entropy. In Section 3 we introduce Lyapunov exponents and some results on the dynamic of the geodesic flow. We also prove theorems \ref{thm1} and \ref{thm2}.\noindent

\begin{ack}
I am very grateful to my Ph.D. adviser Barbara Schapira for many helpful discussions from the very beginning of this work. I would also like to thank Samuel Tapie for his inspiring remarks on some results of this paper, Godofredo Iommi for some nice suggestions that improved the presentation of this work, and the referee of an early version of this work for all the remarks realized.
\end{ack}

\section{Entropy}\label{s:preliminaries}

Let $(X,\mathcal{B},\mu)$ be a probability space and consider a measurable transformation $T:X\to X$ preserving the measure $\mu$. Recall that $\mu$ is said to be \emph{ergodic} if every $T$-invariant measurable set $A\subset X$ verifies $\mu(A)\in\{0,1\}$.

\subsection{Measure-theoretic entropy} Let $\mathcal{P}$ be a countable measurable partition of $X$. The entropy of $\mathcal{P}$ with respect to $\mu$, denoted by $H_{\mu}(\mathcal{P})$, is defined as
$$H_{\mu}(\mathcal{P})=-\sum_{P\in\mathcal{P}}\mu(P)\log\mu(P).$$
For all $n\geq 0$, define the partition $\mathcal{P}^{n}$ as the measurable partition consisting of all possible intersections of elements of $T^{-i}\mathcal{P}$, for all $i=0,...,n-1$. The entropy of $T$ with respect to the partition $\mathcal{P}$ is then defined as the limit
$$h_{\mu}(T,\mathcal{P})=\lim_{n\to\infty} \frac{1}{n}H_{\mu}(\mathcal{P}^{n}).$$
The \textit{measure-theoretic entropy of} $T$, \textit{with respect to} $\mu$, is the supremum of the entropies $h_{\mu}(T,\mathcal{P})$ over all measurable finite partitions $\mathcal{P}$ of $X$, i.e.
$$h_{\mu}(T)=\sup_{\mathcal{P} \text{ finite}} h_{\mu}(T,\mathcal{P}).$$

\subsection{Katok $\delta$-entropies} Suppose from now on that $(X,d)$ is a complete and separable metric space, $T:X\to X$ is a continuous transformation and $\mu$ is a Borel $T$-invariant probability measure on $X$. For every $n\geq 1$ the dynamical distance $d_{n}$ is defined by
$$d_{n}(x,y)=\max_{0\leq i\leq n-1} d(T^{i}x,T^{i}y), \quad \mbox{for all} \quad x,y\in X.$$
The $(n,r)$-dynamical ball centered at $x$, denoted by $B_{n}(x,r)$, is the ball centered at $x$ of radius $r$ for the dynamical distance $d_{n}$. Note that since $T$ is continuous, the $(n,r)$-dynamical balls are open subsets of $X$. Let $B$ be a subset of $X$. A $(n,r)$-\emph{covering} of $B$ is a covering of $B$ by $(n,r)$-dynamical balls. A $(n,r)$-\emph{separated} set in $B$ is a subset $E$ of $B$ such that for every $x,y\in E$, if $x\neq y$, then $d_{n}(x,y)\geq r$.

\begin{definition} Let $B\subset X$ be any set. 
\begin{enumerate}
  \item[(i)]  The minimal cardinality of a $(n,r)$-covering of $B$ is denoted by $N(n,r,B)$, and
  \item[(ii)] the maximal cardinality of a $(n,r)$-separated set in $B$ is denoted by $S(n,r,B)$.
\end{enumerate}
\end{definition}

\begin{remark} Note that $N(n,r,B)$ may be infinite. However, this is not the case when $B$ is a compact set.
\end{remark}

Lemma \ref{CardCovCardSep} below is classical (see for instance \cite[page 169]{Walters}). It states that if $K\subset X$ is compact, then both $N(n,r,K)$ and $S(n,r,K)$ are comparable.

\begin{lemma}\label{CardCovCardSep} Let $n\geq 1$ and $r>0$. Then for all compact subsets $K\subset X$, we have
$$N(n,r,K)\leq S(n,r,K) \leq N(n,r/2,K).$$
\end{lemma}

Recall that Bowen's definition of topological entropy of a continuous transformation on a compact metric space is the following:
\begin{eqnarray*}
h_{top}(T)&=&\lim_{r\to 0}\limsup_{n\to\infty} \frac{1}{n}\log N(n,r,X)\\
&=& \lim_{r\to 0}\limsup_{n\to\infty} \frac{1}{n}\log S(n,r,X).
\end{eqnarray*}

From the measure point of view, Katok proposed the following definition of entropy in \cite{Katok}. For every $0<\delta<1$, denote by $N_\mu(n,r,\delta)$ the minimal cardinality of a $(n,r)$-covering of a set of $\mu$-measure greater than $1-\delta$. Observe that this number is finite since every compact subset of measure greater than $1-\delta$ admits a finite $(n,r)$-covering.

\begin{definition} Let $0<\delta<1$. The \emph{lower} and \emph{upper} \emph{$\delta$-entropies} relative to $\mu$, denoted respectively by $\underline{h}^{\delta}_{\mu}(T)$ and $\overline{h}^{\delta}_{\mu}(T)$, are defined as
$$\underline{h}^{\delta}_{\mu}(T)=\lim_{r\to 0}\liminf_{n\to\infty}\frac{1}{n}\log N_\mu(n,r,\delta)$$
and
$$\overline{h}^{\delta}_{\mu}(T)=\lim_{r\to 0}\limsup_{n\to\infty}\frac{1}{n}\log N_\mu(n,r,\delta).$$
\end{definition}

\begin{proposition}\label{ComparisonKatokEnt} Let $0<\delta_{2}\leq \delta_{1}<1$, then
$$\underline{h}^{\delta_{1}}_{\mu}(T)\leq \underline{h}^{\delta_{2}}_{\mu}(T) \quad \mbox{and} \quad \overline{h}^{\delta_{1}}_{\mu}(T) \leq \overline{h}^{\delta_{2}}_{\mu}(T).$$
\end{proposition}
\begin{proof} We define $\mathcal{B}_{i}$, for $i=1,2$, by $\mathcal{B}_{i}=\{B: \mu(B)>1-\delta_{i}\}$. Since $\mathcal{B}_{2}\subset \mathcal{B}_{1}$, we obtain
\begin{eqnarray*}
\underline{h}^{\delta_{1}}_{\mu}(T)&=& \lim_{r\to 0}\liminf_{n\to\infty}\frac{1}{n}\log \min\{N(n,r,B):B\in \mathcal{B}_{1}\}\\
&\leq& \lim_{r\to 0}\liminf_{n\to\infty}\frac{1}{n}\log \min\{N(n,r,B):B\in \mathcal{B}_{2}\}\\
&=&\underline{h}^{\delta_{2}}_{\mu}(T)
\end{eqnarray*}
The other inequality can be proved similarly.
\end{proof}

When $X$ is a compact metric space Katok proved that the lower and upper $\delta$-entropies are equal and coincide with the measure-theoretic entropy \cite[Theorem 1.1]{Katok}. In his proof the assumption of compactness for $X$ is only used to show that $\overline{h}^{\delta}_{\mu}(T)\leq h_{\mu}(T)$. The other inequality is only based on the fact that $h_{\mu}(T)$ can be approximate by entropies $h_{\mu}(T,\mathcal{P})$, with respect to a partition $\mathcal{P}$ satisfying $\mu(\partial \mathcal{P})=0$, where $\partial \mathcal{P}$ is the union of the boundaries of the elements of $\mathcal{P}$\footnote{The validity of this inequality in the noncompact case also has been remarked in \cite{GK} to compute the topological entropy of the geodesic flow in the modular surface.}. So one can conclude the following:

\begin{theorem}[Katok]\label{KatokEntMTEntropy} Let $X$ be a complete and separable metric space and let $T:X\to X$ be a continuous transformation. If $\mu$ is an ergodic $T$-invariant probability measure, then for all $0<\delta<1$, we have
$$h_{\mu}(T)\leq\underline{h}^{\delta}_{\mu}(T).$$
\end{theorem}

\subsection{Local entropies of Brin-Katok} The aim of this subsection is to understand some relations between previous notions of entropy and local entropy. The notion of local entropy was introduced by Brin and Katok in \cite{BK}.

\begin{definition} The \emph{lower} and \emph{upper} \emph{local entropies} of $T$ relative to $\mu$, denoted respectively by $\underline{h}^{loc}_{\mu}(T)$ and $\overline{h}^{loc}_{\mu}(T)$, are defined as
$$\underline{h}^{loc}_{\mu}(T)=\essinf_{x\in X}\lim_{r\to 0}\liminf_{n\to\infty}-\frac{1}{n}\log \mu(B_{n}(x,r))$$
and
$$\overline{h}^{loc}_{\mu}(T)=\esssup_{x\in X}\lim_{r\to 0}\limsup_{n\to\infty}-\frac{1}{n}\log \mu(B_{n}(x,r)).$$
\end{definition}

\begin{lemma}\label{LemmaEntLocal1} Let $X$ be a complete and separable metric space and let $T:X\to X$ be a continuous transformation. If $\mu$ is an ergodic $T$-invariant probability measure on $X$, then
$$\underline{h}^{loc}_{\mu}(T)=\int \lim_{r\to 0}\liminf_{n\to\infty}-\frac{1}{n}\log \mu(B_{n}(x,r)) d\mu(x)$$
and
$$\overline{h}^{loc}_{\mu}(T)=\int \lim_{r\to 0}\limsup_{n\to\infty}-\frac{1}{n}\log \mu(B_{n}(x,r)) d\mu(x).$$
\end{lemma}
\begin{proof} For the sake of simplicity, for each $x\in X$ we denote by $\underline{h}^{loc}_{\mu}(T,x)$ the limit
$$\lim_{r\to 0}\liminf_{n\to\infty}-\frac{1}{n}\log \mu(B_{n}(x,r)).$$
Since $T(B_{n+1}(x,r))\subset B_{n}(Tx,r)$ for every $n\geq 1$, the $T$-invariance of $\mu$ implies that
\begin{equation}\label{eq:LemmaEntLocal1}
\underline{h}^{loc}_{\mu}(T,Tx)\leq \underline{h}^{loc}_{\mu}(T,x),
\end{equation}
for every $x\in X$. Define $\underline{\eta}(x)$ as $\underline{\eta}(x)=\inf_{k\geq 0} \underline{h}^{loc}_{\mu}(T,T^k x)$. By definition $\underline{\eta}$ is a $T$-invariant function, so it is $\mu$-a.e. constant equal to some constant $\underline{\eta}(\mu)$. Note that $\underline{\eta}(\mu)$ is also equal to the essential infimum of $\underline{h}^{loc}_{\mu}(T,x)$. On the other hand, inequality \ref{eq:LemmaEntLocal1} implies that
$$\underline{\eta}(x)=\lim_{n\to\infty}\frac{1}{n}\sum_{k=0}^{n-1} \underline{h}^{loc}_{\mu}(T,T^k x),$$
for every $x\in X$. Using Birkhoff's Ergodic Theorem, we conclude that
$$ \underline{\eta}(\mu) = \int \underline{h}^{loc}_{\mu}(T,x) d\mu(x),$$
which is exactly the first desired equality. The second equality follows from the same strategy by considering the supremum instead of the infimum in every involved term.
\end{proof}

 As in the case of $\delta$-entropies, the lower and upper local entropies coincide and are equal to the measure-theoretic entropy. This was proved in \cite{BK} for continuous transformations defined on compact metric spaces. However, again only one inequality in the proof requires the compactness assumption on $X$. Indeed, the other inequality follows from from the fact that for ergodic probability measures and finite partitions having boundaries with null measure, the inequality $h_\mu(T,\mathcal{P})\leq \underline{h}^{loc}_{\mu}(T,x)$ is verified for almost every $x\in X$. So, we obtain

\begin{theorem}[Brin-Katok]\label{InfLocalEntropyMTEntropy} Let $X$ be a complete and separable metric space and let $T:X\to X$ be a continuous transformation. If $\mu$ is an ergodic $T$-invariant probability measure on $X$, then
\begin{eqnarray}\label{Entropyandlocalentropy1}
h_{\mu}(T)\leq \underline{h}^{loc}_{\mu}(T).
\end{eqnarray}
\end{theorem}

Following Ledrappier \cite{Ledrappier2}, we now prove that \ref{Entropyandlocalentropy1} is an equality for Lipschitz transformations defined on (noncompact) Riemannian manifolds.

\begin{theorem}\label{LedrappierEqualityLocEntropy} Let $T:M\to M$ be a Lipschitz transformation of a complete Riemannian manifold and $\mu$ an ergodic $T$-invariant probability measure. Then
\begin{gather*}\label{Entropyandlocalentropy2}
h_{\mu}(T)= \underline{h}^{loc}_{\mu}(T).
\end{gather*}
\end{theorem}
\begin{proof} As consequence of Theorem \ref{InfLocalEntropyMTEntropy}, we only need to prove that $h_{\mu}(T)\geq\underline{h}^{loc}_{\mu}(T)$. This follows from Proposition \ref{KeyPropLed} below.

\begin{proposition}[Ledrappier \cite{Ledrappier2}, Proposition 6.3]\label{KeyPropLed} Let $T:M\to M$ be a Lipschitz transformation of a complete Riemannian manifold and $\mu$ an ergodic $T$-invariant probability measure. Then, for every compact set $K\subset M$ such that $\mu(K)>0$ and all $0<r<1$, there exists a partition $\widehat{\mathcal{P}}$ of $K$ with finite entropy such that, if $\mathcal{P}=\widehat{\mathcal{P}}\cup\{M\setminus K\}$, then for $\mu$-almost every $x\in K$ the sequence $(n_{k})_{k\geq 0}$ of return times of $x$ into $K$ satisfies
\begin{gather*}\label{LELE1}
\mathcal{P}^{n_{k}}(x)\subset B_{n_{k}}(x,r),
\end{gather*}
for all $k\geq 0$.
\end{proposition}

Let $\mathcal{P}$ be as in Proposition \ref{KeyPropLed}. Using the ergodicity of $\mu$, for $\mu$-a.e. $x\in M$ there exists an integer $k>0$ such that $T^{-k}x\in K$. We know from the construction of $\mathcal{P}$ that the inclusion $\mathcal{P}^{n}(T^{-k}x)\subset B_{n}(T^{-k}x,r)$ is satisfied for infinitely many integers $n$. In particular, we deduce that $\mathcal{P}^{n+k}(x)\subset T^{k}B_{n}(T^{-k}x,r)$ is also satisfied for infinitely many $n$'s. Therefore
\begin{eqnarray*}
\limsup_{n\to\infty}-\frac{1}{n}\log \mu(\mathcal{P}^{n+k}(x))&\geq & \liminf_{n\to\infty}-\frac{1}{n}\log \mu(T^{k}B_{n}(T^{-k}x,r))\\
&=& \liminf_{n\to\infty}-\frac{1}{n}\log \mu(B_{n}(T^{-k}x,r))\\
&\geq& \liminf_{n\to\infty}-\frac{1}{n}\log \mu(B_{n-k}(x,r)).
\end{eqnarray*}
Hence, using Shannon-McMillan-Breiman Theorem, we get for $\mu$-a.e. $x\in M$
\begin{gather}\label{LELE0}
\lim_{n\to\infty}-\frac{1}{n}\log \mu(\mathcal{P}^{n}(x))\geq \liminf_{n\to\infty}-\frac{1}{n}\log \mu(B_{n}(x,r)).
\end{gather}
By Lemma \ref{LemmaEntLocal1} and Monotone Convergence Theorem, for every $\varepsilon>0$ there exists $r_{0}>0$ such that for all $0<r<r_{0}$, we have
$$\int \liminf_{n\to\infty}-\frac{1}{n}\log \mu(B_{n}(x,r)) d\mu(x) \geq \underline{h}^{loc}_{\mu}(T)-\varepsilon.$$
Using inequality \ref{LELE0} and once more Shannon-McMillan-Breiman Theorem, we obtain
\begin{eqnarray*}
h_{\mu}(T)&\geq& h_{\mu}(T,\mathcal{P})=\int \lim_{n\to\infty} -\frac{1}{n}\log \mu(\mathcal{P}^{n}(x))d\mu(x)\\
&\geq& \int \liminf_{n\to\infty}-\frac{1}{n}\log \mu(B_{n}(x,r)) d\mu(x)\geq \underline{h}^{loc}_{\mu}(T)-\varepsilon.\\
\end{eqnarray*}
Finally, the desired inequality follows when $\varepsilon\to 0$.
\end{proof}

In some cases, the study of the ergodic theory of a dynamical system is simpler by restricting ourselves to the dynamic on compact subsets. In this direction, for Lipschitz transformations we get to describe the measure-theoretic entropy as a sort of ``maximal local entropy on compact sets''.

\begin{theorem}\label{thm01} Let $T:M\to M$ be a Lipschitz transformation of a complete Riemannian manifold and $\mu$ an ergodic $T$-invariant probability measure on $M$. Then
\begin{gather*}\label{eq:thm01}
h_{\mu}(T)= \sup_{K}\esssup_{x\in K}\lim_{r\to 0}\limsup_{\stackrel{n\to\infty}{T^{n}x\in K}}-\frac{1}{n}\log \mu(B_{n}(x,r)),
\end{gather*}
where the supremum is taken over all the compact subsets $K$ of $M$ having positive $\mu$-measure.
\end{theorem}
\begin{proof}
On the one hand, Theorem \ref{InfLocalEntropyMTEntropy} says that
$$h_{\mu}(T)\leq \underline{h}^{loc}_{\mu}(T),$$
and, on the other hand, by definition of local entropy, we have
\begin{equation}\label{eq:sol_thm01}
\underline{h}^{loc}_{\mu}(T) \leq \sup_{K}\esssup_{x\in K}\lim_{r\to 0}\limsup_{\stackrel{n\to\infty}{T^{n}x\in K}} -\frac{1}{n}\log \mu(B_{n}(x,r)).
\end{equation}
We claim that the right-hand side of inequality \ref{eq:sol_thm01} is less than $h_{\mu}(T)$. Indeed, since $T$ is Lipschitz, Proposition \ref{KeyPropLed} and Shannon-McMillan-Breiman Theorem imply that, for every compact set $K\subset M$ such that $\mu(K)>0$, and every $0<r<1$, we have
\begin{gather*}\label{ModifiedLocalEntropy2}
\limsup_{\stackrel{n\to\infty}{T^{n}x\in K}} -\frac{1}{n}\log \mu(B_{n}(x,r))\leq h_\mu(T,\mathcal{P}) \leq h_\mu(T),
\end{gather*}
for $\mu$-a.e. $x\in K$. Clearly this implies the desired inequality, which concludes the proof of Theorem \ref{thm01}.
\end{proof}

\begin{remark} Theorems \ref{InfLocalEntropyMTEntropy} and \ref{thm01} seem similar but are different in a key point. Theorem \ref{InfLocalEntropyMTEntropy} deals with the lower local entropy whereas Theorem \ref{thm01} deals with the upper local entropy.
\end{remark}

To end this subsection, and for completeness of this whole section, we give a relation between the upper $\delta$-entropy and the upper local entropy in a general setting.

\begin{theorem}\label{KatokEntLocEntThm} Let $T:X\to X$ be a continuous transformation of a complete and separable metric space $(X,d)$ and $\mu$ an ergodic $T$-invariant probability measure. Then for all $0<\delta<1$, we have
$$\overline{h}^{\delta}_{\mu}(T) \leq \overline{h}^{loc}_{\mu}(T).$$
\end{theorem}
\begin{proof} Fix $\varepsilon>0$ and $0<r<1$. Define the set $X(\varepsilon,r,m)\subset X$, for $m\geq 1$, by
$$X(\varepsilon,r,m)=\left\{x\in X : -\frac{1}{n}\log \mu(B_{n}(x,r/2))\leq \overline{h}^{loc}_{\mu}(T)+\varepsilon, \textmd{ for all } n\geq m\right\}.$$
Note that $\mu(X(\varepsilon,r,m))$ goes to $1$ when $m\to\infty$ and $r\to 0$. Take $r>0$ small enough and $m_{0}=m_{0}(r)>0$ large enough such that $\mu(X(\varepsilon,r,m))>1-\delta$ for every $m\geq m_{0}$. Let $K\subset X(\varepsilon,r,m_{0})$ be a compact set such that $\mu(K)>1-\delta$. We are going to find an upper bound of $S(n,r,K)$ (the maximal cardinality of a $(n,r)$-separated subset of $K$), for every $n\geq m_{0}$. Let $E$ be a maximal $(n,r)$-separated set in $K$. Since $(n,r/2)$-balls centered at $E$ are disjoint, we have
$$\sum_{x\in E}\mu(B_{n}(x,r/2))=\mu\left(\bigcup_{x\in E}B_{n}(x,r/2)\right)\leq 1.$$
As the $(n,r/2)$-balls with center in $K$ satisfy $\mu(B_{n}(x,r/2))\geq \exp\left(-n\left(\overline{h}^{loc}_{\mu}(T)+\varepsilon\right)\right)$, it follows that
$$\# E = S(n,r,K)\leq \exp\left(n\left(\overline{h}^{loc}_{m}(T)+\varepsilon\right)\right).$$
Therefore, by Lemma \ref{CardCovCardSep}, we get
\begin{eqnarray*}
\limsup_{n\to \infty}\frac{1}{n}\log N(n,r,K)&\leq& \limsup_{n\to \infty}\frac{1}{n}\log S(n,r,K)\\
&\leq& \overline{h}^{loc}_{\mu}(T)+\varepsilon.
\end{eqnarray*}
Hence, for every $r>0$ small enough we can find a compact $K\subset X$ such that $\mu(K)>1-\delta$ and
$$\limsup_{n\to \infty}\frac{1}{n}\log N(n,r,K) \leq \overline{h}^{loc}_{\mu}(T)+\varepsilon.$$
In particular,
\begin{eqnarray*}
\overline{h}^{\delta}_{\mu}(T)&=& \lim_{r\to 0}\limsup_{n\to \infty}\frac{1}{n}\log N(n,r,\delta)\\
&\leq& \lim_{r\to 0}\limsup_{n\to \infty}\frac{1}{n}\log N(n,r,K)\\
&\leq& \overline{h}^{loc}_{\mu}(T)+\varepsilon.
\end{eqnarray*}
Since $\varepsilon>0$ is arbitrary, the conclusion follows.
\end{proof}

\subsection{A Riemannian local entropy for Riemannian manifolds}

Our goal now is to define a geometric entropy of a transformation with respect to an invariant probability measure $\mu$. We do this by measuring the Riemannian volume $\mathcal{L}$ of $\mu$-typical dynamical balls. Moreover, we want to be able to compare it with the measure-theoretic entropy. It turns out that the essential supremum of the exponential decay for the Riemannian measure of $\mu$-typical dynamical balls is the ``good'' quantity to consider.

\begin{definition}\label{localRiemannianEntropyDef} Let $T:M\to M$ be a continuous transformation of a complete Riemannian manifold, preserving an ergodic $T$-invariant probability measure $\mu$. For every compact set $K\subset M$ verifying $\mu(K)>0$, we define the \textit{local Riemannian entropy of} $T$ relative to $\mu$ over $K$, denoted by $h^{\mathcal{L}}_{\mu}(T,K)$, as
$$h^{\mathcal{L}}_{\mu}(T,K)=\esssup_{x\in K}\lim_{r\to 0}\limsup_{\stackrel{n\to\infty}{T^{n}x\in K}} -\frac{1}{n}\log \mathcal{L}(B_{n}(x,r)),$$
where the essential supremum is with respect to the measure $\mu$. We define the \textit{local Riemannian entropy of} $T$ relative to $\mu$, denoted by $h^{\mathcal{L}}_{\mu}(T)$, as
$$h^{\mathcal{L}}_{\mu}(T)=\sup_{K}h^{\mathcal{L}}_{\mu}(T,K),$$
where the supremum is taken over all the compact subsets $K$ of $M$ having positive measure with respect to $\mu$.
\end{definition}

The following theorem shows that the lower $\delta$-entropy of Katok is bounded from above by the local Riemannian entropy. We stress the fact that the measure $\mu$ only appears in the definition of local Riemannian entropy when considering $\mu$-typical dynamical balls. For the best of our knowledge there are no related results in the literature.

\begin{theorem}\label{localRiemannianEntropyIneq} Let $T:M\to M$ be a continuous transformation of a complete Riemannian manifold, preserving an ergodic $T$-invariant probability measure $\mu$. If $K\subset M$ is a compact set of strictly positive $\mu$-measure, then for all $1-\mu(K)^{2}<\delta<1$, we have
$$\underline{h}^{\delta}_{\mu}(T)\leq h^{\mathcal{L}}_{\mu}(T,K).$$
\end{theorem}

\begin{proof} If $h^{\mathcal{L}}_{\mu}(T,K)=\infty$ there is nothing to prove. Suppose that $h^{\mathcal{L}}_{\mu}(T,K)<\infty$. For $\varepsilon>0$, $r>0$ and $m\geq 1$, we define the set $K_{\varepsilon,r,m}$ as
\begin{eqnarray*}
K_{\varepsilon,r}=\{x\in K &:& \mathcal{L}(B_{n}(x,r))\geq \exp(-n(h^{\mathcal{L}}_{\mu}(T,K)+\varepsilon)), \textmd{ for every } n\geq m\\
&& \textmd{such that } T^{n}x\in K\}.
\end{eqnarray*}
Note that the measure $\mu(K_{\varepsilon,r,m})$ goes to $\mu(K)$ when $m\to\infty$ and $r\to 0$. For all $0<\eta<\mu(K)/2$ there exist $r>0$ and $m_{0}\geq 1$ (depending on $r$) such that $\mu(K_{\varepsilon,r,m_{0}})>\mu(K)-\eta/2$. Let $K_{0}\subset K_{\varepsilon,r,m_{0}}$ be a compact set with measure $\mu(K_{0})> \mu(K)-\eta$. We are going to estimate the cardinality of a minimal $(n,r)$-covering of $K_{0}$ for $n\geq m_0$. The problem is that in general, if $x,x'\in K$ are different, the first time of return in $K$ is also different for these two points. The ergodicity assumption for the dynamical system will help us to erase this problem.

Birkhoff's Ergodic Theorem implies that $\frac{1}{n}\sum_{i=0}^{n-1}\mu(K_{0}\cap T^{-i}K_{0})$ converge to $\mu(K_{0})^{2}$. In particular, there is a sequence $(\phi(n))_{n}$ strictly increasing of integers such that $\mu(K_{0}\cap T^{-\phi(n)}K_{0})$ converge to $L(K_{0})\geq \mu(K_{0})^{2}$. Let $0<\lambda<L(K_{0})/2$. Then, there is an integer $n_0\geq 1$ such that $\mu(K_{0}\cap T^{-\phi(n)}K_{0})>L(K_{0})-\lambda$ for all $n\geq n_0$. Let $\delta(K_{0},\lambda)=1-(\mu(K_{0})^{2}-\lambda)$ and set $K_{\phi(n)}=K_{0}\cap T^{-\phi(n)}K_{0}$. The $\mu$-measure of $K_{\phi(n)}$ satisfies, for all $n\geq n_0$
$$\mu(K_{\phi(n)})>L(K_{0})-\lambda \geq \mu(K_{0})^{2}-\lambda=1-\delta(K_{0},\lambda).$$

Let $E$ be a maximal set $(\phi(n),r)$-separated in $K_{\phi(n)}$, for $n\geq \max\{m_0,n_0\}$. If $V_r(K)$ denotes the open $r$-neighbourhood of $K$, then
\begin{eqnarray*}
\mathcal{L}(V_{r}(K))&\geq& \mathcal{L}\left(\bigcup_{x\in E} B_{\phi(n)}(x,r/2)\right)\\
&\geq& \sum_{x\in E} \mathcal{L}(B_{\phi(n)}(x,r/2))\\
&\geq& \# E \exp(-n(h^{\mathcal{L}}_{\mu}(T,K)+\varepsilon)).
\end{eqnarray*}
Therefore, the cardinality of $E$ is bounded from above by
$$\# E \leq \mathcal{L}(V_{r}(K))\exp(n(h^{\mathcal{L}}_{\mu}(T,K)+\varepsilon)).$$
Hence, using Lemma \ref{CardCovCardSep} and the estimation from above of the cardinality of a maximal set $(\phi(n),r)$-separated in $K_{\phi(n)}$, we get
\begin{eqnarray*}
\liminf_{n\to\infty} \frac{1}{n}\log N(n,r,\delta(K_{0},\lambda)) &\leq& \liminf_{n\to\infty} \frac{1}{\phi(n)}\log N(\phi(n),r,\delta(K_{0},\lambda))\\
&\leq& \liminf_{n\to\infty} \frac{1}{\phi(n)}\log N(\phi(n),r,K_{\phi(n)})\\
&\leq& \liminf_{n\to\infty} \frac{1}{\phi(n)}\log S(\phi(n),r,K_{\phi(n)})\\
&\leq& h^{\mathcal{L}}_{\mu}(T,K)+\varepsilon.
\end{eqnarray*}
In particular, we have shown that for every $r>0$ small enough and every $n$ large enough (depending on $r$), there exists a compact set $K_{\phi(n)}$ of $\mu$-measure $\mu(K_{\phi(n)})\geq \delta(K_{0},\lambda)$. Therefore, the sequence of inequalities above implies that
\begin{eqnarray*}
\underline{h}^{\delta(K_{0},\lambda)}_{\mu}(T)&=& \lim_{r\to 0}\liminf_{n\to\infty} \frac{1}{n}\log N(n,r,\delta(K_{0},\lambda))\\
&\leq& h^{\mathcal{L}}_{\mu}(T,K)+\varepsilon.
\end{eqnarray*}
Since $\lambda>0$ is arbitrary and from Proposition \ref{ComparisonKatokEnt}, we have $\underline{h}^{\delta}_{\mu}(T)\leq h^{\mathcal{L}}_{\mu}(T,K)+\varepsilon$ for all $1-\mu(K_{0})^{2}<\delta<1$. Since $\eta>0$ is arbitrary, we have $\underline{h}^{\delta}_{\mu}(T)\leq h^{\mathcal{L}}_{\mu}(T,K)+\varepsilon$, for all $1-\mu(K)^{2}<\delta<1$. Since $\varepsilon>0$ is arbitrary, the conclusion of the theorem follows.
\end{proof}
\noindent
A direct consequence of Theorem \ref{localRiemannianEntropyIneq}, by choosing a sequence of compact sets $(K_n)_n$ such that $\mu(K_n)\to 1$ when $n\to 1$, is the following corollary.
\begin{corollary}\label{cor1localRiemannianEntropyIneq} Let $T:M\to M$ be a continuous transformation of a complete Riemannian manifold, preserving an ergodic $T$-invariant probability measure $\mu$. Then, for all $0<\delta<1$, we have
$$\underline{h}^{\delta}_{\mu}(T)\leq h^{\mathcal{L}}_{\mu}(T).$$
\end{corollary}

By Theorem \ref{KatokEntMTEntropy} we know that the measure-theoretic entropy is less than every $\delta$-entropy. This fact together with Corollary \ref{cor1localRiemannianEntropyIneq} imply 

\begin{corollary}\label{cor0} Let $T:M\to M$ be a continuous transformation of a complete Riemannian manifold, preserving an ergodic $T$-invariant probability measure $\mu$. Then
\begin{equation*}\label{eq:thm0}
h_{\mu}(f)\leq h^{\mathcal{L}}_{\mu}(T).
\end{equation*}
\end{corollary}

We have considered the Riemannian measure in the definition of local Riemannian entropy because we can always ensure that it gives positive measure to $\mu$-typical dynamical balls, regardless of the measure $\mu$. However, we might also have considered the measure $\mu$ as in the case of local entropies, and nothing in the proof of Corollary \ref{cor1localRiemannianEntropyIneq} changes. Hence, for all $0<\delta<1$, we have
\begin{equation*}\label{ModifiedLocalEntropy}
\underline{h}^{\delta}_{\mu}(T) \leq \sup_{K}\esssup_{x\in K}\lim_{r\to 0}\limsup_{\stackrel{n\to\infty}{T^{n}x\in K}} -\frac{1}{n}\log \mu(B_{n}(x,r)).
\end{equation*}

\begin{theorem}\label{thm:delta_mtentropy} Let $T:M\to M$ be a Lipschitz transformation of a complete Riemannian manifold and $\mu$ an ergodic $T$-invariant probability measure. Then, for all $0<\delta<1$, we have
$$h_\mu(T)=\underline{h}^{\delta}_{\mu}(T).$$
\end{theorem}
\begin{proof}
On the one hand, Theorem \ref{KatokEntMTEntropy} implies that $h_\mu(T)\leq\underline{h}^{\delta}_{\mu}(T)$. On the other hand, by last inequality in the proof of Theorem \ref{thm01}, we conclude that $\underline{h}^{\delta}_{\mu}(T)\leq h_\mu(T)$.
\end{proof}

We now can now summarize theorems \ref{LedrappierEqualityLocEntropy}, \ref{thm01} and \ref{thm:delta_mtentropy} in one single statement.

\begin{corollary}
Let $T:M\to M$ be a Lipschitz transformation of a complete Riemannian manifold and $\mu$ an ergodic $T$-invariant probability measure. Then
$$h_\mu(T)=\underline{h}^{\delta}_{\mu}(T)=\underline{h}^{loc}_{\mu}(T)=\sup_{K}\esssup_{x\in K} \lim_{r\to 0}\limsup_{\stackrel{n\to\infty}{T^{n}x\in K}} -\frac{1}{n}\log \mu(B_{n}(x,r)).$$
\end{corollary}

\section{Geodesic flow and Lyapunov exponents}

\subsection{Basic concepts} Let $X$ be a complete Riemannian manifold with dimension at least 2 and pinched negative sectional curvature (that is $-b^2\leq \kappa\leq -a^2<0$, where $0<a<b$). Let $T^{1}X$ be its unit tangent bundle. Recall that the Liouville measure $\mathcal{L}$ on $T^1X$ is the Riemannian volume of the Sasaki metric on $T^{1}X$. It is invariant under the action of the geodesic flow $(g^{t})$ on $T^{1}X$. Let $v\in T^{1}X$ and $t\in \R$. Denote by $E^{su}(v)$ the tangent space at $v$ of the corresponding strong unstable manifold\footnote{We refer to \cite{Ballmann} for proper definitions of the objects involved in this section and to \cite{PPS} for details in the noncompact setting.}. Denote by $J^{su}(v,t)$ the Jacobian of the linear map $d_{v}g^{t}|_{E^{su}(v)}$. The \emph{geometric potential} $F^{su}:T^{1}X\to \R$ is then defined by
$$F^{su}(v)=-\frac{d}{dt}\bigg{|}_{t=0}\log J^{su}(v,t).$$

\begin{theorem}[Paulin-Pollicott-Schapira \cite{PPS}, Theorem 7.1]\label{PotGeom} Let $X$ be a complete Riemannian mani\-fold with dimension at least 2 and pinched negative sectional curvature. Assume that the derivatives of the sectional curvature are uniformly bounded. Then $F^{su}$ is H\"older-continuous and bounded.
\end{theorem}

For $v\in T^1X$, let $\|\cdot\|_{v}$ denote the Riemannian norm on $T_{v}T^1X$. The vector $v$ is said to be (Lyapunov-Perron) \textit{regular} if there exist numbers $\{\lambda_{i}(v)\}_{i=1}^{l(v)}$, called \textit{Lyapunov exponents}, and a decomposition of the tangent space at $v$ into $T_{v}T^1X=\bigoplus_{i=1}^{l(v)}E_{i}(v)$ such that for every tangent vector $w\in E_{i}(v)\setminus\{0\}$, we have
$$\lim_{n\to\pm\infty}\frac{1}{n}\log \|d_{v}g^{n}(w)\|_{g^{n}v}=\lambda_{i}(v),$$
and
$$\lim_{n\to\pm\infty}\frac{1}{n}\log|\det(d_v g^n)|= \sum_{i=1}^{l(v)}\lambda_i(v)\Dim(E_i(v)).$$
Let $\Lambda$ be the set of regular vectors. By Oseledec's Theorem (see \cite{Oseledec}, \cite{Ledrappier3}) and comparison theorems in Riemannian geometry, if $\mu$ is any $g$-invariant probability measure on $T^1X$ (where $g=g^1$ is the time-one map of the geodesic flow), then the set $\Lambda$ has full $\mu$-measure. Moreover, the functions $v\mapsto \lambda_{j}(v)$ and $v\mapsto \dim(E_{j}(v))$ are $\mu$-measurable and $g$-invariant. In particular, if $\mu$ is ergodic, they are $\mu$-a.e. constant. In that case, we denote by $\{\lambda_{j}\}_{j=1}^{l}$ the set of Lyapunov exponents. 

Let $v\in \Lambda$. Note that by hyperbolicity, we get
$$E^{su}(v)=\bigoplus_{\lambda_{j}(v)>0} E_{j}(v).$$
Consider now the function $\chi^{+}: T^1X\to \R$ defined as 
$$\chi^{+}(v)=\sum_{\lambda_j(v)>0} \lambda_j(v)\dim(E_j(v))$$ 
if $v\in\Lambda$ and $\chi^{+}(v)=0$ otherwise. If $\mu$ is an ergodic $g$-invariant probability measure on $T^1X$, we denote by $\chi^{+}(\mu)$ (or simply $\chi^{+}$ when there is no ambiguity) the essential value of the function $\chi^{+}(v)$ with respect to $\mu$. 

The potential $F^{su}$ is intimately related to the Lyapunov exponents. Indeed, using the notation above, if $\mu$ is a $(g^t)$-invariant probability measure on $T^{1}X$, then
\begin{equation*}
\lim_{n\to+\infty}\frac{1}{n}\int_{0}^{n}F^{su}(g^{t}v)dt=-\lim_{n\to+\infty}\frac{1}{n}\log J^{su}(v,n)=-\chi^{+}(v), \quad \mu-\text{a.e}.
\end{equation*}
\noindent

The key fact that will allow us to prove Ruelle's inequality for the geodesic flow is the Gibbs property of the Liouville measure for the potential $F^{su}$ (see \cite[Proposition 7.9]{PPS}). Recall that a $(g^{t})$-invariant measure $m$ on $T^{1}X$ satisfies the Gibbs property for the potential $F:T^{1}X\to \R$ with constant $c(F)$ if and only if for every compact subset $K$ of $T^{1}X$, for every $r>0$, there exists $C=C(K,r)\geq 1$, such that for every $T\geq 0$, for every $v\in K\cap g^{-T}K$, we have
\begin{equation*}
C^{-1} \leq \dfrac{m(B_{T}(v,r))}{\exp(\int_{0}^{T}(F(g^{t}v)-c(F))dt)}\leq C.
\end{equation*}

\begin{proposition}[Paulin-Pollicott-Schapira \cite{PPS}, Proposition 7.9]\label{gibbsprop} Let $X$ be a complete Riemannian manifold with dimension at least 2 and pinched negative sectional curvature. Assume that the derivatives of the sectional curvature are uniformly bounded. Then the Liouville measure on $T^{1}X$ satisfies the Gibbs property for the potential $F^{su}$ and the constant $c(F^{su})=0$.
\end{proposition}

Note that the assumption on the derivatives of the sectional curvature is crucial. It implies in particular that the strong unstable and strong stable distributions are H\"older-continuous (see for instance \cite[Theorem 7.3]{PPS}), and therefore that $\mathcal{L}$ locally decomposes into a product of Lebesgue measures along unstable and stable manifolds (see \cite[Theorem 7.6]{PPS}). This last fact is the cornerstone in order to estimate the Liouville measure of a dynamical ball.

\subsection{Ruelle's inequality and Pesin's formula}\label{ss:RuellePesin} Let $X$ be a complete Riemannian manifold satisfying the assumptions of Theorem \ref{thm1}. By simplicity we will always consider in the proofs an ergodic $(g^{t})$-invariant probability measure $\mu$. The proofs of the theorems in the non-ergodic case are consequence of the ergodic decomposition of such a measure. We can also assume that $g=g^{1}$ is an ergodic transformation with respect to $\mu$. If it is not the case, then we can choose an ergodic time $\tau$ for $\mu$ (see \cite[Theorem 3.2]{LoSc}) and prove Theorem \ref{thm1} and Theorem \ref{thm2} for the map $g^{\tau}$. The validity of Theorem \ref{thm1} and Theorem \ref{thm2} for $g^{\tau}$ implies the validity of both theorems for $g$ since $h_{\mu}(g^{\tau})=\tau h_{\mu}(g)$ and the Lyapunov exponents of $g^{\tau}$ are $\tau$-multiples of the Lyapunov exponents for $g$.

\begin{proof}[Proof of Theorem \ref{thm1}] Let $K\subset T^{1}X$ be a compact set of measure $0<\mu(K)<1$. Since $F^{su}$ is $\mu$-integrable, Proposition \ref{gibbsprop} implies
\begin{eqnarray*}
h^{\mathcal{L}}_{\mu}(g,K) &=& \esssup_{v\in K}\lim_{r\to 0}\limsup_{\stackrel{n\to\infty}{g^{n}v\in K}} -\frac{1}{n}\log \mathcal{L}(B_{n}(v,r))\\
&\leq& \esssup_{v\in K}\lim_{r\to 0}\limsup_{\stackrel{n\to\infty}{g^{n}v\in K}} -\frac{1}{n}\log \left(C^{-1}\exp\left(\int_{0}^{n}F^{su}(g^{t}v)dt\right)\right)\\
&=& -\int F^{su} d\mu \\
&=& \chi^{+}(\mu).
\end{eqnarray*}
Last equality is a consequence of Birkhoff's Ergodic Theorem. Thus, $h^{\mathcal{L}}_{\mu}(g)\leq \chi^{+}$ and Corollary \ref{cor0} implies directly that
$$h_{\mu}(g)\leq \chi^{+}.$$
\end{proof}

The proof of Theorem \ref{thm2} is similar to those in \cite{LS}, \cite{Ledrappier} and \cite{LY}. We only need to corroborate that all technical hypotheses hold, for instance the H\"older regularity of strong unstable and strong stable distributions. As said before, these technical hypotheses are consequence of the assumption on the derivatives of the sectional curvature. In \cite{OP} the authors use the regularity of the strong unstable foliation to prove the existence of nice measurable partitions. They follow the ideas in \cite{LS} and \cite{Ledrappier} adapted to the geodesic flow in negative curvature. We stress that in \cite{OP} the authors use the H\"older regularity of strong unstable and strong stable foliations omitting the hypothesis on the derivatives of the sectional curvatures, even if it is necessary to ensure such regularity (we refer to \cite{BBB} where the authors construct a finite volume Riemannian surface with pinched negative sectional curvatures whose strong stable foliation is not H\"older-continuous).\\

Recall that a measurable partition $\xi$ of $T^{1}X$ is subordinate to the $W^{su}$-foliation if for $\mu$-a.e. $v\in T^{1}X$, we have
\begin{enumerate}
  \item[(i)] the atom $\xi(v)$ is contained in $W^{su}(v)$, and
  \item[(ii)] the atom $\xi(v)$ contains a neighborhood of $v$, open for the submanifold topology on $W^{su}(v)$.
\end{enumerate}
Let $\vol_{v}$ be the volume on $W^{su}(v)$ induced by the Sasaki metric on $T^1 X$ restricted to the strong unstable manifold $W^{su}(v)$. The measure $\mu$ has absolutely continuous conditional measures on unstable manifolds if for every $\mu$-measurable partition $\xi$ subordinate to $W^{su}$, the conditional measure $\mu_{\xi(v)}$ of $\mu$ on $\xi(v)$ is absolutely continuous with respect to $\vol_{v}$.

\begin{proposition}\label{keyLedStrLedYoung} Let $X$ be a complete Riemannian manifold with dimension at least 2 and pinched negative sectional curvature. Assume that the derivatives of the sectional curvature are uniformly bounded. Let $\mu$ an ergodic $(g^{t})$-invariant probability measure and suppose that $g=g^{1}$ is ergodic. Then, there exists a $\mu$-measurable partition $\xi$ of $T^{1}X$, such that
\begin{enumerate}
  \item[(1)] the partition $\xi$ is decreasing, i.e. $(g^{-1}\xi)(v)\subset\xi(v)$ for $\mu$-a.e. $v\in T^1 X$,
  \item[(2)] the partition $\bigvee_{n\geq 0} g^{-n}\xi$ is the partition into points,
  \item[(3)] the partition $\xi$ is subordinate to the $W^{su}$-foliation,
  \item[(4)] for $\mu$-a.e. $v$, we have $\bigcup_{n\in\Z}g^{n}\xi(g^{-n}v)=W^{su}(v)$,
  \item[(5)] for all measurable sets $B\subset T^{1}X$, the map
  $$\psi_{B}(v)=\mathrm{vol}_{v}(\xi(v)\cap B)$$
  is measurable and $\mu$-a.e. finite,
  \item[(6)] for $\mu$-a.e. $v\in T^{1}X$, if $w,w'\in \xi(v)$, then the infinite product
  $$\Delta(w,w')=\prod_{n=0}^{\infty} \dfrac{J^{su}(g^{-n}w,1)}{J^{su}(g^{-n}w',1)}$$
  converges, and
  \item[(7)] there exist constants $C>0$ and $0<\alpha<1$ such that, if $w\in\xi(v)$, then
  $$|\log\Delta(v,w)|\leq C(d(v,w))^{\alpha}.$$
\end{enumerate}
\end{proposition}
\begin{proof}
The existence of $\mu$-measurable partitions satisfying $(1)-(4)$ is proved in \cite{OP}. Properties $(5)-(7)$ are consequence of the regularity of the strong unstable distribution and the regularity of $J^{su}$, following the same proof of \cite[Proposition 3.1]{Ledrappier}.
\end{proof}

The class of $\mu$-measurable partitions satisfying $(1)-(4)$ contains somehow all the complexity of the dynamics of the geodesic flow in the sense that every partition in this class maximises the measure-theoretic entropy. This result is proved in \cite{OP} following the ideas in \cite{Ledrappier} and \cite{LY}.

\begin{proposition}[Ledrappier-Young/Otal-Peign\'e]\label{L/OP} Let $X$ be a complete Riemannian manifold with dimension at least 2 and pinched negative sectional curvature. Assume that the derivatives of the sectional curvature are uniformly bounded. Let $\mu$ be an ergodic $(g^{t})$-invariant probability measure and suppose that $g=g^{1}$ is ergodic. If $\xi$ is a partition as in Proposition \ref{keyLedStrLedYoung}, then
$$h_{\mu}(g)=h_{\mu}(g,\xi),$$
where
$$h_{\mu}(g,\xi):=H_{\mu}(g^{-1}\xi|\xi)=-\int \log \mu_{\xi(v)}((g^{-1}\xi)(v))d\mu(v).$$
\end{proposition}

\begin{proof}[Proof of Theorem \ref{thm2}] We remark that the computation of the entropy appears in \cite{LS}, but as this fact is not stated explicitly, we give the general idea behind. Suppose that $\mu$ has absolutely continuous conditional measures on unstable manifolds. Let $\xi$ be a $\mu$-measurable partition as in Proposition \ref{keyLedStrLedYoung}. We only have to prove that $h_{\mu}(g,\xi)=\chi^{+}$. This is equivalent to show that $H_{\mu}(g^{-1}\xi|\xi)=\int \log J^{su}(v,1) d\mu(v)$.\\
Define the measure $\nu$ on $T^{1}X$ by
$$\nu(B)=\int \vol_{w}(\xi(w)\cap B) d\mu(w),$$
for every measurable subset $B$ of $T^{1}X$. Property (5) in Proposition \ref{keyLedStrLedYoung} implies that $\nu$ is $\sigma$-finite. Since $\mu_{\xi(v)}$ is absolutely continuous with respect to $\vol_{v}$, the measure $\mu$ is absolutely continuous with respect to $\nu$. Moreover, the Radon-Nikodym derivative $\kappa=d\mu/d\nu$ coincide with $d\mu_{\xi(v)}/d\vol_{v}$, $\vol_{v}$-almost everywhere on $\xi(v)$, for $\mu$-almost every $v\in T^{1}X$ (see \cite[Proposition 4.1]{LS}). Thus,
$$-\log \mu_{\xi(v)}((g^{-1}\xi)(v))=-\log \int_{(g^{-1}\xi)(v)}\kappa(w)d\vol_{v}(w).$$
Using Change of Variables Theorem, it follows
$$\int_{(g^{-1}\xi)(v)}\kappa(w)d\vol_{v}(w)=\int _{\xi(gv)} \kappa(g^{-1}w)\frac{1}{J^{su}(g^{-1}w,1)}d\vol_{gv}(w).$$
From \cite[Proposition 4.2]{LS}, the application $L(w)=\frac{\kappa(w)}{\kappa(g^{-1}w)} J^{su}(g^{-1}w,1)$ is constant on the atoms of the partition $\xi$. Therefore,
\begin{eqnarray*}
\int _{\xi(gv)} \kappa(g^{-1}w)\frac{1}{J^{su}(g^{-1}w,1)}d\vol_{gv}(w)&=& \int _{\xi(gv)} \frac{\kappa(w)}{L(w)}d\vol_{gv}(w)\\
&=& \frac{1}{L(gv)} \int_{\xi(gv)}\kappa(w)d\vol_{gv}(w)\\
&=& \frac{1}{L(gv)} \int_{\xi(gv)} d\mu_{\xi(gv)}(w)\\
&=& \frac{1}{L(gv)}.
\end{eqnarray*}
Putting all together, we have shown that 
\begin{equation}\label{eq:pesin1}
-\log \mu_{\xi(v)}((g^{-1}\xi)(v))=\log J^{su}(v,1)+\log\frac{\kappa(gv)}{\kappa(v)}.
\end{equation}
Since the left hand side in \ref{eq:pesin1} is non-negative and $\log J^{su}(v,1)$ is $\mu$-integrable, it follows that the negative part of $\log\frac{\kappa(gv)}{\kappa(v)}$ is $\mu$-integrable. In particular, its $\mu$-integral is equal to zero (see \cite[Proposition 2.2]{LS}), thus
$$h_{\mu}(g)=-\int \log \mu_{\xi(v)}((g^{-1}\xi)(v))d\mu(v) = \int \log J^{su}(v,1)d\mu(v)=\chi^{+}.$$

The converse statement is just the conclusion of \cite[Theorem 3.4]{Ledrappier} under the hypothesis obtained in Proposition \ref{L/OP}, for a $\mu$-measurable partition $\xi$ as in Proposition \ref{keyLedStrLedYoung}.
\end{proof}

\subsection{Further comments} We discuss now some consequences of Theorem \ref{thm1} in thermodynamic formalism. The \emph{topological pressure of} $(g^{t})$ \emph{for a potential} $F:T^{1}X\to \R$, denoted by $P_g(F)$ (or simply $P(F)$), is defined as
$$P(F)=\sup_{\mu} P(F,\mu),$$
where $P(F,\mu)=h_{\mu}(g)+\int_{T^{1}X} F d\mu$ and $\mu$ is an $(g^t)$-invariant probability measure on $T^1 X$. A $(g^t)$-invariant probability measure $m$ on $T^1 X$ is said to be an \emph{equilibrium state for} $F$, if
$$P(F)= P(F,m).$$
In \cite{PPS} the authors construct a Gibbs measure for every bounded H\"older-continuous potential $F$, with constant $c(F)$ equal to the topological pressure $P(F)$. We remark that if a Gibbs measure is finite, its normalization is the unique equilibrium state for the potential, and if infinite, there is no equilibrium state for $F$ (see \cite[Theorem 6.1]{PPS}). Note that, as a consequence of Theorem \ref{PotGeom}, there exists a Gibbs measure for $F^{su}$ under the hypotheses of Theorem \ref{thm1}, which is denoted by $m_{F^{su}}$.

Observe now that, in terms of thermodynamical formalism, Ruelle's inequality can be stated as
\begin{corollary}\label{cor0} Let $X$ be a complete Riemannian manifold with dimension at least 2 and pinched negative sectional curvature. Assume that the derivatives of the sectional curvature are uniformly bounded. Then, for every $(g^{t})$-invariant probability measure $\mu$ on $T^1 X$, we have
\begin{equation}\label{eq:PRuelle}
P(F^{su},\mu)\leq 0.
\end{equation}
\end{corollary}

In particular, we can remove inequality \ref{eq:PRuelle} as a redundant assumption in \cite[Theorem 7.2]{PPS} and obtain Corollary \ref{cor1} below. Recall that the geodesic flow is conservative with respect to a finite or infinite measure $m$ on $T^1X$ if every wandering set has $m$-measure zero.

\begin{corollary}\label{cor1} Let $X$ be a complete Riemannian manifold with dimension at least 2 and pinched negative sectional curvature. Assume that the derivatives of the sectional curvature are uniformly bounded. If the geodesic flow on $T^{1}X$ is conservative with respect to the Liouville measure $\mathcal{L}$, then $\mathcal{L}$ is proportional to the Gibbs measure $m_{F^{su}}$ associated to the geometric potential $F^{su}$. Furthermore, the topological pressure $P(F^{su})$ is equal to zero.
\end{corollary}
As a direct consequence, we also have
\begin{corollary}\label{cor2} Let $X$ be a complete Riemannian manifold with dimension at least 2 and pinched negative sectional curvature. Assume that the derivatives of the sectional curvature are uniformly bounded. If $X$ has finite volume, then
$$\frac{m_{F^{su}}}{m_{F^{su}}(T^1 X)}=\frac{\mathcal{L}}{\mathcal{L}(T^{1}X)}.$$
\end{corollary}

Finally, we remark that Corollary \ref{cor:summary_Ruelle} follows directly from all the statements in this last subsection.

\end{document}